\documentclass[reqno,oneside,12pt]{amsart}

\usepackage[T1]{fontenc}
\usepackage{times,mathptm}
\usepackage{amssymb,epsfig,verbatim,xypic}
\def\vs{\vspace{0.1cm}}

\theoremstyle{plain}

\newtheorem{thm}{Theorem}

\newtheorem{cor}[thm]{Corollary}
\newtheorem{pro}[thm]{Proposition}
\newtheorem{lem}[thm]{Lemma}
\newtheorem{proposition-principale}[thm]{Proposition principale}
\newtheorem{thm-principal}{Th\'eor\`eme principal}[section]

\theoremstyle{definition}

\newtheorem{eg}[thm]{Example}
\newtheorem{rem}[thm]{Remark}

\newenvironment{thm-A}
{{\vs \noindent \bf Theorem A.$\,$}\it}{\vs}

\newenvironment{thm-B}
{{\vs \noindent \bf Theorem B.$\,$}\it}{\vs}

\newenvironment{thm-BB}
{{\vs \noindent \bf Theorem B'.$\,$}\it}{\vs}

\def\C{\mathbf{C}}

\def\Z{\mathbf{Z}}

\def\bfk{{\mathbf{k}}}
\def\bfF{{\mathbf{F}}}

\newcommand{\bbA}{{\mathbb{A}}}

\newcommand{\Spec}{{\mathrm{Spec}}}

\newcommand{\sT}{{\mathcal T}}

\newcommand{\id}{{\rm id}}

\def\P{\mathbb{P}}

\def\Aut{{\sf{Aut}}}
\def\Bir{{\sf{Bir}}}

\def\PGL{{\sf{PGL}}}

\def\GL{{\sf{GL}}}

\def\Pic{{\mathrm{Pic}}}

\def\Ind{{\text{Ind}}}
\def\Exc{{\text{Exc}}}

\addtocounter{section}{0}             

\usepackage{color}

\begin{document}

\setlength{\baselineskip}{0.53cm}        
%
%
\title[]
{Birational conjugacies between endomorphisms on the projective plane}
\date{2019}

\author{Serge Cantat}
\address{Serge Cantat, IRMAR, Campus de Beaulieu,
b\^atiments 22-23
263 avenue du G\'en\'eral Leclerc, CS 74205
35042  RENNES C\'edex}
\email{serge.cantat@univ-rennes1.fr}

\author{Junyi Xie}
\address{Junyi Xie, IRMAR, Campus de Beaulieu,
b\^atiments 22-23
263 avenue du G\'en\'eral Leclerc, CS 74205
35042  RENNES C\'edex}
\email{junyi.xie@univ-rennes1.fr}

\thanks{The last-named author is partially supported by project ``Fatou'' ANR-17-CE40-0002-01, the first-named author by the french academy of sciences (fondation del Duca). }

%
%

\maketitle
 

%
%
%
%




{\noindent}{\bf{1. The statement. --}} Let $\bfk$ be an algebraically closed field of characteristic $0$. If $f_1$ and 
$f_2$ are two endomorphisms of a projective surface $X$ over $\bfk$ and  $f_1$ is conjugate to $f_2$
by a birational transformation of $X$, then $f_1$ and $f_2$ have the same topological degree. 
When $X$ is the projective plane $\P^2_\bfk$, $f_1$ (resp. $f_2$) is given by homogeneous formulas of the same degree $d$ without common 
factor, and $d$ is called the degree, or algebraic degree of $f_1$; in that case the topological degree is $d^2$, so, $f_1$ and $f_2$ have the same degree $d$ if they are conjugate. 

\smallskip

\begin{thm-A} 
Let $\bfk$ be an algebraically closed field of characteristic $0$.  Let $f_1$ and $f_2$ be dominant endomorphisms of $\P^2_\bfk$. 
Let $h:\P^2_\bfk\dashrightarrow \P^2_\bfk$ be a birational map such that $h\circ f_1=f_2\circ h$.
If the degree $d$ of $f_1$ is $\geq 2$, there exists an automorphism $h': \P^2_\bfk\to \P^2_\bfk$ such that $h'\circ f_1=f_2\circ h'$. 

Moreover, 
$h$ itself is in $\Aut(\P^2_\bfk)$, except maybe if $f_1$ is conjugate by an element of $\Aut(\P^2_\bfk)$  to

\begin{enumerate}
\item  the composition of $g_d: [x:y:z]\mapsto [x^d:y^d:z^d]$ and a permutation of the coordinates,

\item or the endomorphism $ (x,y) \mapsto (x^d, y^d +\sum_{j=2}^d a_j y^{d-j})$ of the open subset 
$\bbA^1_\bfk\setminus \{0\}\times \bbA^1_\bfk\subset \P^2_\bfk$, for some coefficients $a_j\in \bfk$.
\end{enumerate}
\end{thm-A}

\smallskip

Theorem~A is proved in Sections~2 to 6. 
In Section~7 we describe centralizers of endomorphisms of $\P^2_\bfk$ of degree $\geq2$. 
A counter-example to Theorem~A is given in Section~8 when ${\mathrm{char}}(\bfk)\neq 0$.
The case $d=1$ is covered by~\cite{Blanc2006}; in particular, there are automorphisms $f_1,f_2\in \Aut(\P^2_{\bfk})$ which are conjugate by some birational transformation but not by an automorphism.

\begin{eg}
When $f_1=f_2$ is the composition of $g_d$ and a permutation of the coordinates and $h$ is the Cremona involution $[x:y:z]\mapsto [x^{-1}:y^{-1}:z^{-1}]$, we have $h\circ f_1=f_2\circ h.$
\end{eg}

This example illustrates what may happen in Case (1) of Theorem~A.  Case~(2) is treated in details in Section~5: Lemma~\ref{lem-mixed} describes more precisely what may happen in this case. 

The exceptional examples from Cases~(1) and (2) preserve pencils of lines: we refer to~\cite{DJ1, DJ2, FP} for a study of rational transformations of the plane preserving a web of curves. 


\smallskip

{\noindent}{\bf{Acknowledgement. --}} Theorem~A answers a question of T. Gauthier and G. Vigny in dimension $2$. We thank them 
for sharing their ideas. We also thank D.-Q. Zhang for answering our questions on the theorem of R. V. Gurjar, and D. Villalobos Paz and J. Moraga and the anonymous referee for pointing out mistakes in the first version of Section~8, and of case (2) of Theorem~A. 

\medskip

{\noindent}{\bf{2. The exceptional locus. --}}
If $h:\P^2\dashrightarrow \P^2$ is a birational map, we denote by $\Ind(h)$ its {\bf{indeterminacy locus}} (a finite subset of 
$\P^2(\bfk)$), and by $\Exc(h)$ its {\bf{exceptional set}}, i.e. the union of the curves contracted by $h$ (a finite union of irreducible curves). 
Let $U_h=\P^2_\bfk\setminus \Exc(h)$ be the complement of $\Exc(h)$; it is a Zariski dense open subset of $\P^2_\bfk$.
If $C\subset \P^2_\bfk$ is a curve, we denote by $h_\circ(C)$ the {\bf{strict transform}} of $C$, i.e. the Zariski closure of $h(C\setminus \Ind(f))$.

\begin{pro}\label{prouhreg} If $h$ is a birational transformation of the projective plane, then 
(1) $\Ind(h)\subseteq \Exc(h)$, (2) $h|_{U_h}(U_h)=U_{h^{-1}}$, and  (3) $h|_{U_h}: U_h\to U_{h^{-1}}$ is an isomorphism. 
\end{pro}
\begin{proof}
There is a smooth projective surface $X$ and two birational morphisms $\pi_1,\pi_2: X\to \P^2$ such that $h=\pi_2\circ\pi_1^{-1}$;
we choose $X$ minimal, in the sense that there is no $(-1)$-curve $C$ of $X$ which is contracted by both $\pi_1$ and~$\pi_2$ (\cite{Hartshorne1977}).

Pick a point $p\in \Ind(h)$. The divisor $\pi_1^{-1}(p)$ is a tree of rational curves of negative self-intersections, 
with at least one $(-1)$-curve. If $p\notin \Exc(h)$, any curve  contracted by $\pi_2$ that intersects $\pi_1^{-1}(p)$ is in fact contained in $\pi_1^{-1}(p)$.
But $\pi_2$ may be decomposed as a succession of contractions of $(-1)$-curves: since it does not contract any $(-1)$-curve in $\pi_1^{-1}(p)$, 
we deduce that $\pi_2$ is a local isomorphism along $\pi_1^{-1}(p)$. Then, we would obtain a $(-1)$-curve in $\P^2_\bfk$, contradicting the minimality of $\P^2_\bfk$; hence $\Ind(h)\subset \Exc(h)$.
Thus  $h|_{U_h}: U_h\to \P^2$ is regular.  Since $U_h\cap \Exc(h)=\emptyset,$ $h|_{U_h}$ is an open immersion,  $h^{-1}$ is well defined on $h|_{U_h}(U_h)$, and $h^{-1}$ is an open immersion on $h|_{U_h}(U_h)$. It follows that $h|_{U_h}(U_h)\subseteq U_{h^{-1}}.$
The same argument shows that $h^{-1}|_{U_{h^{-1}}}: U_{h^{-1}}\to \P^2$ is well defined and its image is in $U_h.$
Since $h^{-1}|_{U_{h^{-1}}}\circ h|_{U_h}=\id$ and $h|_{U_h}\circ h^{-1}|_{U_{h^{-1}}}=\id$; this concludes the proof. 
\end{proof}

\medskip

Let $f_1$ and $f_2$ be dominant endomorphisms of $\P^2_\bfk$. Let $h:\P^2\dashrightarrow \P^2$ be a birational map such that $f_1=h^{-1}\circ f_2\circ h$.
Let $d$ be the common (algebraic) degree of $f_1$ and $f_2$.
Recall that an algebraic subset $C$ of $\P^2_\bfk$ is {\bf{totally invariant}} 
under the action of an endomorphism $g$ if $g^{-1}(C)=C$; moreover, if  $g^{-1}(C)=C$ and $C$ is non-empty, then $g$ must be dominant, $g(C)=C$, and if $\deg(g)\geq 2$, $g$ ramifies along $C$\footnote{
That $g$ be dominant follows from the following fact: 
{\emph{the image of an endomorphism of $\P^2$ is equal to $\P^2$ or to a singleton}}. Indeed, in homogeneous coordinates, $g$ is given by three homogeneous polynomials of the same degree $d\geq 0$ without common factor of positive degree. If $d=0$, $g(\P^2)$ is a singleton. 
If $d\geq 1$, the preimage of a point is finite, by Bezout theorem.}.

\begin{lem}\label{lemtotinv} The exceptional set of $h$ is totally invariant under the action of $f_1$: $f_1^{-1}(\Exc(h))=\Exc(h)$.
\end{lem}
\begin{proof} Since $h\circ f_1= f_2\circ h$, the strict transform of $f_1^{-1}(\Exc(h))$ by $f_2\circ h$ is a finite set, but every dominant endomorphism of $\P^2_\bfk$ is a finite map, so the strict transform of $f_1^{-1}(\Exc(h))$ by $h$ is already a finite set. This means that $f_1^{-1}(\Exc(h))$ is contained in 
$\Exc(h)$;  this implies $f_1(\Exc(h))\subset \Exc(h)$ and then $f_1^{-1}(\Exc(h))=\Exc(h)=f_1(\Exc(h))$ because $f_1$ is onto.
\end{proof}

\begin{lem}\label{lemconfexc} If $d\geq 2$ then $\Exc(h)$ and $\Exc(h^{-1})$ are two isomorphic configurations of lines, and this
configuration falls in the following list:
\begin{enumerate}
\item[(P0)] the empty set;
\item[(P1)] one line in $\P^2$;
\item[(P2)]  two lines in $\P^2$;
\item[(P3)]  three lines in $\P^2$ in general position.
\end{enumerate}
\end{lem}

\begin{proof} Assume $\Exc(h)$ is not empty; then, by Lemma \ref{lemtotinv},  the curve $\Exc(h)$ is totally invariant under $f_1$.
According to \cite[\S 4]{Fornaess-Sibony} and~\cite[Proposition 2]{Cerveau-LN}, $\Exc(h)$ is one of the three curves listed in (P1) to (P3).

Changing $h$ into $h^{-1}$ and permuting the role of $f_1$ and $f_2$, we see that $\Exc(h^{-1})$ is also a configuration of type (P$i$) for some $i$.
Proposition \ref{prouhreg} shows that $U_h\simeq U_{h^{-1}}$.
Since the four possibilities (P$i$) correspond to pairwise 
non-isomorphic complements, we deduce that $\Exc(h)$ and $\Exc(h^{-1})$ have the same type. 
\end{proof}

\begin{rem}
One can also refer to \cite{Gurjar2003} to prove this lemma. Indeed,  
$f_1$ induces a map from the set of irreducible components of $\Exc(h)$ into itself, 
and since $f_1$ is onto, this map is a permutation; the same applies to $f_2$. Thus,
replacing $f_1$ and $f_2$ by $f_1^m$ and $f_2^m$ for some suitable $m\geq 1$, we may assume that $f_1(C)=C$ for every irreducible component $C$ of $\Exc(h)$.
Since $f_1$ is finite, $\Exc(h)$ has only finitely many irreducible components, and $f_1(\Exc(h))=\Exc(h)$, we obtain $f_1^{-1}(C)=C$ for every component. Since $f_1$ acts by multiplication by $d$ on $\Pic(\P^2_\bfk)$, the ramification index of $f_1$ along $C$ is $d>1$, and
the main theorem of \cite{Gurjar2003} implies that $C$ is a line. 
\end{rem}

\begin{rem}
{\emph{Totally invariant hypersurfaces of endomorphisms of $\P^3$  are unions of hyperplanes, at most four of them}}. We refer to~\cite{Horing:2017} for a proof and important additional 
references, notably the work of J.-M. Hwang, N. Nakayama and D.-Q. Zhang; see also the recent paper of Y. Mabed on this subject~\cite{Mabed}, and the references therein. So, an analog of Lemma~5 holds
in dimension $3$ too; but our proof in case (P1), see \S~4 below, does not apply in dimension $3$, at least not directly. (Note that~\cite{Briend-Cantat-Shishikura:2004} contains an important gap, 
since its main result is based on an incorrect lemma from~\cite{Briend-Duval:2001}).
\end{rem}

\medskip

{\noindent}{\bf{3. Normal forms. --}}
Two configurations of the same type (P$i$) are equivalent under the action of $\Aut(\P^2_\bfk)=\PGL_3(\bfk)$. 
If we change $h$ into $A\circ h\circ B$ for some well chosen pair of automorphisms $(A,B)$, or equivalently if we change $f_1$ into $B\circ f_1\circ B^{-1}$ and 
$f_2$ into $A^{-1}\circ f_2\circ A$, we may assume that $\Exc(h)=\Exc(h^{-1})$ and that exactly  one of the following situation occurs (see also~\cite{Fornaess-Sibony}):

\smallskip

{{\bf{(P0).--}} {$\Exc(h)=\Exc(h^{-1})=\emptyset$.--}\indent Then $h$ is an automorphism of $\P^2_\bfk$ and Theorem~A is proved. 

\smallskip

{{\bf{(P1).--}} {$\Exc(h)=\Exc(h^{-1})=\{z=0\}$.--}\indent  Then $h$ induces an automorphism of $\bbA^2_\bfk$ and
$f_1$ and $f_2$ restrict to endomorphisms of $\bbA^2_\bfk=\P^2_\bfk\setminus \{z=0\}$ (that extend to endomorphisms of $\P^2_\bfk$).

\smallskip

{{\bf{(P2).--}} {$\Exc(h)=\Exc(h^{-1})=\{x=0\}\cup \{z=0\}$.--}\indent  
Then, 
$U_h$ and $U_{h^{-1}}$ are both equal to the open set $U:=\{(x,y)\in \bbA^2|\,\, x\neq 0\}$. Moreover,
\begin{equation}\label{eq:h-jonq}
h|_U(x,y) = (Ax^{\pm 1}, Bx^{m}y+C(x))
\end{equation}
for some regular function $C(x)$ on $\bbA^1_\bfk\setminus\{ 0\}$ and $m\in \Z$, and 
\begin{equation}\label{eqf_ip2}
f_i|_U(x,y)= (x^{\pm d}, F_i(x,y))
\end{equation} 
for some rational functions $F_i\in \bfk(x)[y]$ which are regular on $(\bbA^1_\bfk\setminus\{Â 0\})\times \bbA^1$
and have degree $d$ (more precisely, $f_i$ must define an endomorphism of $\P^2$ of degree $d$).
Moreover, the signs of the exponent $\pm d$ in Equation~\eqref{eqf_ip2} are the same for $f_1$ and $f_2$.

\smallskip

{{\bf{(P3).--}} {$\Exc(h)=\Exc(h^{-1})=\{x=0\}\cup\{y=0\}\cup \{z=0\}$.--}\indent  
In this case, 
each $f_i$ is equal to $a_i\circ g_d$ where  $g_d( [x:y:z])= [x^d:y^d:z^d]$
and each $a_i$  is an automorphism of $\P^2_\bfk$ acting by permutation of the coordinates, while 
$h$ is an automorphism of $(\bbA^1\setminus \{0\})\times (\bbA^1\setminus \{0\})$.

\medskip

{\noindent}{\bf{4. Endomorphisms of $\bbA^2_\bfk$. --}}
This section proves Theorem~A in case (P1):  

\begin{pro}\label{propolyendo} Let $f_1$ and $f_2$ be endomorphisms  of $\bbA^2$ that extend to endomorphisms of $\P^2$ of degree $d\geq 2$. If $h$ is an automorphism of $\bbA^2$ that conjugates $f_1$ to $f_2$ then  $h$ is an affine automorphism i.e. $\deg h=1$.
\end{pro}

We follow the notation from \cite{Favre2011} and denote by 
$V_{\infty}$ the valuative tree of $\bbA^2=\Spec ( \bfk[x,y] )$ at infinity.
If $g$ is an endomorphism of $\bbA^2$, we denote by $g_\bullet$ its action on $V_{\infty}$.
 
Set $V_1=\{v\in V_{\infty} \; ; \; \alpha(v)\geq 0, A(v)\leq 0\}$, where $\alpha$ 
and $A$ are respectively the skewness and thinness function, as defined in page 216 of
\cite{Favre2011}; the set $V_1$  is a closed subtree of $V_{\infty}$. For $v\in V_1$, $v(F)\leq  0$ for every $F\in \bfk[x,y]\setminus \{0\}$. 
Then $V_1$ is invariant under each $({f_i})_{\bullet}$, and if we 
set 
\begin{equation}
\sT_i=\{v \in V_1 \; ; \; ({f_i})_{\bullet}v=v\}
\end{equation} 
then $\sT_2=h_{\bullet}\sT_1$. Since each $f_i$ extends to an endomorphism of $\P^2_\bfk$, the valuation  $-\deg$ is an element of $\sT_1\cap \sT_2$. Also, in the terminology of \cite{Favre2011}, $\lambda_2(f_i)=\lambda_1(f_i)^2=d^2$ and $\deg(f_i^n)=\lambda_1^n=d^n$ 
for all $n\geq 1$ and for $i=1$ and $2$, because $f_1$ and $f_2$ extend to regular endomorphisms of $\P^2_\bfk$ of degree $d$. So
by \cite[Proposition 5.3 (a)]{Favre2011}, $\sT_i$ is a single point or a closed segment.  

\smallskip 

A valuation $v\in V_{\infty}$ is {\bf{monomial}}  of weight $(s,t)$ for the pair of polynomial functions $(P,Q)\in \bfk[x,y]^2$ if 
\begin{enumerate}
\item $P$ and $Q$ generate $\bfk[x,y]$ as a $\bfk$-algebra,
\item if $F$ is any non-zero element of $\bfk[x,y]$ and $F=\sum_{i,j\geq 0}a_{ij}P^iQ^j$ is its decomposition as a polynomial function
of  $P$ and $Q$ then
\begin{equation}
v(F)=-\max\{si+tj \; ; \; a_{i,j}\neq 0\}.
\end{equation}
\end{enumerate}
We say that $v$ is monomial for the basis  $(P,Q)$ of $\bfk[x,y]$, if  $v$ is monomial for $(P,Q)$ and some weight $(s,t)$.
In particular, $-\deg$ is monomial for $(x,y)$, of weight $(1,1).$

\begin{lem}\label{lemprovow}If $v\in V_1$ is monomial for $(P,Q)$ of weight $(s,t)$, then   $s,t\geq 0$, and $\min\{s,t\}=\min\{-v(F) \; ; \; F\in \bfk[x,y]\setminus \bfk\}.$
\end{lem}
\begin{proof}
First, assume that $(P,Q)=(x,y)$. For an element $v$ of $V_1$, $v(F)\leq 0$ for every $F$ in $\bfk[x,y]$, hence $s=-v(x)$ and $t=-v(y)$
are non-negative; and the formula for $\min\{s,t\}$ follows from the inequality $-v(F)\geq \min\{s,t\}$. To get the statement for any pair $(P,Q)$, 
change $v$ into $g^{-1}_\bullet v$ where $g$ is the automorphism defined by $g(x,y)=(P(x,y),Q(x,y))$.
\end{proof}

\begin{lem}\label{lemdegmon}If $-\deg$ is monomial for $(P,Q)$, of weight $(s,t)$, then $s=t=1$ and $P$ and $Q$ are of degree one in $\bfk[x,y].$
\end{lem}

\begin{proof}By Lemma \ref{lemprovow}, we may assume that $1=s\leq t$; thus, 
after an affine change of variables, we may assume that $P=x.$ Since $\bfk[x,y]$ is generated by $x$ and $Q$, $Q$ takes form 
$Q=ay+C(x)$ where $a\in \bfk^*$ and $C\in \bfk[x]$. 
If $C$ is a constant, we conclude the proof. Now we assume $\deg (C)\geq 1$.
Then $t=\deg(Q)=\deg(C)$. Since $y=a^{-1}(Q-C(x))$ and $-\deg$ is monomial for $(x,Q)$ of weight $(1,t)$, we get $1=\deg(y)=\max\{t, \deg C\}=t$. It follows that $t=\deg Q=1,$ which concludes the proof.
\end{proof}

\proof[Proof of Proposition \ref{propolyendo}]
By \cite[Proposition 5.3 (b), (d)]{Favre2011}, there exist $P$ and $Q\in \bfk[x,y]$ such that 
for every $v\in \sT_1$, $v$ is monomial for $(P,Q)$. Moreover, $-\deg$ is in $\sT_1\cap \sT_2$. By Lemma~\ref{lemdegmon},  $P=x$ and $Q=y$ after an affine change of coordinates.
Since $\sT_2=h_{\bullet} \sT_1$, for every $v\in \sT_2$, $v$ is monomial for $(h^*x,h^*y).$ Since $-\deg\in \sT_2$,  Lemma \ref{lemdegmon} implies $\deg h^*x=\deg h^*y=1$ and this concludes the proof.
\endproof

\medskip

{\noindent}{\bf{5. Endomorphisms of $(\bbA^1_\bfk\setminus \{0 \})\times \bbA^1_\bfk$. --}}
We now arrive at case (P2), namely $\Exc(h)=\Exc(h^{-1})=\{x=0\}\cup \{z=0\}$, and keep the notation from Section~4.
Our first goal is to prove the following lemma.

\begin{lem}\label{lem-mixed}
If $h$ is not an automorphism of $\P^2$,
then after  conjugacy of $f_1$ and $f_2$  by affine transformations of the 
plane, we are in one of the following cases:

\begin{enumerate}
\item $f_1$ and $f_2$ are equal to $(x^d,y^d)$ and $h(x,y)=(Ax^{\pm 1}, Bx^my)$ with $A$ and $B$  
roots of unity of order dividing $d-1$ and $m\in \Z\setminus\{0\}$.

\item Up to a permutation of $f_1$ and $f_2$, 
$$
f_1(x,y) =(x^d, y^d +\sum_{j=2}^d a_j y^{d-j}) \;  {\text{ and }} \; f_2(x,y) =(x^d, y^d +\sum_{j=2}^d a_j (B/A)^j x^{j} y^{d-j})
$$
with $a_j\in \bfk$,
 and $h(x,y)=(Ax, Bxy)$ with $A$ and $B$ two 
roots of unity of order dividing $d-1$; then $h'[x:y:z]=[Az/B:y:x]$ is an automorphism of $\P^2$
that conjugates $f_1$ to $f_2$.

\item Up to a permutation of $f_1$ and $f_2$, 
$$
f_1(x,y) =(x^d, y^d +\sum_{j=2}^d a_j x^{j} y^{d-j}) \;  {\text{ and }} \; f_2(x,y) =(x^d, y^d +\sum_{j=1}^d a_j (A/B)^j x^{j} y^{d-j})
$$
with $a_j\in \bfk$ and $h(x,y)=(Ax^{-1}, Bx^{-2}y)$ with $A$ and $B$ two 
roots of unity of order dividing $d-1$; then $h'[x:y:z]=[(A/B) x: y:z]$ is an automorphism of $\P^2$
that conjugates $f_1$ to $f_2$. 

\item Up to a permutation of $f_1$ and $f_2$, 
$$
f_1(x,y) =(x^d, y^d +\sum_{j=2}^d a_j y^{d-j}) \;  {\text{ and }} \; f_2(x,y) =(x^d, y^d +\sum_{j=1}^d B^j a_j  y^{d-j})
$$
with $c_j\in \bfk$ and $h(x,y)=(Ax^{-1}, By)$ with $A$ and $B$ two 
roots of unity of order dividing $d-1$; then $h'[x:y:z]=[x: y/B:z]$ is an automorphism of $\P^2$
that conjugates $f_1$ to $f_2$. 

\item 
$f_1$ and $f_2$ are equal to $(x^{-d},x^{-d}y^d)$ and $h(x,y)=(Ax, Bx^my)$ with $A$ is a root of unity of order dividing $d+1$ and $B^{d-1}=A^{-1}$ and $m\in \Z\setminus\{0\}$.
\end{enumerate}
\end{lem}

Note, moreover, that the endomorphisms $f_1$ and $f_2$ in (3) have the same form as $f_2$ in (2), hence are conjugate by linear projective automorphisms of $\P^2$ to endomorphisms of the same form as $f_1$ in (2) (these linear projective maps induce automorphisms of $(\bbA^1_\bfk\setminus \{0 \})\times \bbA^1_\bfk$). Thus, this lemma proves Theorem~A in case (P2).
When $f_1=f_2$, case (2) does not appear, case (3) implies that $h(x,y)=(Ax^{-1}, Bx^{-2}y)$ with $A^{d-1}=1=B^{d-1}$ plus the additional constraint that $(A/B)^j=1$ when $a_j\neq 0$, and case (4) implies $B^j=1$ if $a_j\neq 0$.

\begin{proof} We split the proof in two steps.

\smallskip

{\bf{Step 1.--}} 
We first  assume that $f_i|_U(x,y)= (x^{d}, F_i(x,y))$, with $d>0$.

Since $f_i$ extends to a degree $d$ endomorphism of $\P^2_\bfk$, we can write $F_1(x,y)=a_0y^d+\sum_{j=1}^da_j(x)y^{d-j}$ where $a_0\in \bfk^*$  and the $a_j\in \bfk[x]$
satisfy $\deg(a_j)\leq j$ for all $j$.
Changing the coordinates to $(x,by)$ with $b^d=a_0$, we assume $a_0=1$. We can also conjugate $f_1$ by the automorphism 
\begin{equation}
(x,y)\mapsto \left( x,y+\frac{1}{d} a_1(x) \right)
\end{equation}
and assume $a_1=0$. Altogether, the change of coordinates $(x,y)\mapsto (x, by+\frac{1}{d} a_1(x))$ is affine because 
$\deg(a_1)\leq 1$, and conjugates $f_1$ to an endomorphism $(x^{d}, F_1(x,y))$ normalized 
by $F_1(x,y)=y^d + \sum_{j=2}^d a_j(x) y^{d-j}$ with $\deg(a_j)\leq j$. 
Similarly, we may assume
that $F_2(x,y)=y^d+\sum_{j=2}^db_j(x)y^{d-j}$ for some polynomial functions $b_j$ with $\deg(b_j)\leq j$ for all $j$.

Now, with the notation used in Equation~\eqref{eq:h-jonq},
$h(x,y) = (Ax^{\epsilon}, Bx^{m}y+C(x))$, with $\epsilon =\pm 1$, and 
the two terms of the conjugacy relation $h\circ f_1=f_2\circ h$ are
\begin{align}
h\circ f_1 &=(Ax^{\epsilon d}, \; Bx^{dm}(y^d+\sum_{j=2}^da_j(x)y^{d-j})+C(x^d) )\label{eq:conj2}\\
f_2\circ h &=(A^d x^{\epsilon d}, \; (Bx^my+C(x))^d+\sum_{j=2}^db_j(Ax^{\epsilon}) (Bx^my+C(x))^{d-j} ) \label{eq:conj1}.
\end{align}
This gives $A^{d-1}=1$ and comparing the terms of degree $d$ in $y$ we get $B^{d-1}=1$. 
Then, looking at the term of degree $d-1$ in $y$, 
we obtain $C(x)=0.$
Thus $h(x,y)=(Ax^{\epsilon}, Bx^m y)$ for some roots of unity $A$ and $B$, the orders of which divide $d-1$. 
Since $h$ is not an automorphism, we have 
\begin{equation}\label{eq:mneq0}
m \neq 0 \;{\text{ if }}\; \epsilon =1, {\text{ and }}\; m\neq -1\; {\text{ if }}\; \epsilon=-1.
\end{equation}
Coming back to~\eqref{eq:conj2} and~\eqref{eq:conj1}, we obtain the sequence of equalities 
\begin{equation}\label{eq:bjaj}
	b_j(Ax^{\epsilon}) = a_j(x) (Bx^m)^j 
\end{equation}
for all indices $j$ between $2$ and $d$. On the other hand, $a_j$ and $b_j$ are elements of $\bfk[x]$ of degree at most $j$.

\smallskip

\noindent{\bf{Step 1.a.--}} We first treat the case $\epsilon=1$, i.e.\  $h(x,y)=(Ax, Bx^m y)$; then $m\neq 0$.
There are only three possibilities.
\begin{enumerate}
\item[(a)] All $a_j$ and $b_j$ are equal to $0$; then  $f_1(x,y)=f_2(x,y)=(x^d, y^d)$, which concludes the proof.
\item[(b)] Some $a_j$ is different from $0$ and $m\geq 1$. Then, $m=1$, all coefficients $a_j$ are constant, and $b_j(x)=a_j \left(\frac{Bx}{A}\right)^j$
for all indices $j=2, \ldots, d$.
\item[(c)] Some $a_j$ is different from $0$ and $m\leq -1$. Then, $m=-1$, all coefficients $b_j$ are constant, and $a_j(x)=b_j \left({x/B}\right)^j$
for all indices $j=2, \ldots, d$.
\end{enumerate}
Note that (b) and (c) are equivalent after permutating  $f_1$ and $f_2$ (or changing $h$ into $h^{-1}(x,y)=(x/A,y/(ABx))$).

In case (b), we set $\alpha=B/A$ (a root of unity of order dividing $d-1$),  and use homogeneous coordinates to write
\begin{align}
f_1[x:y:z] &=[x^d : y^d +\sum_{j=2}^d a_j z^{j} y^{d-j}: z^d] \label{eqpps}\\ 
f_2[x:y:z] &=[x^d : y^d +\sum_{j=2}^d a_j \alpha^j x^{j} y^{d-j}: z^d] \label{eqpp}.
\end{align}
The conjugacy $h[x:y:z]=[Axz : Bxy : z^2]$ is not a linear projective automorphism of $\P^2$, but
the automorphism defined by 
$[x:y:z]\mapsto \left[ z/\alpha:y:x \right]$ conjugates $f_1$ to $f_2$. A similar computation holds in case (c) (permuting the role of $f_1$ and $f_2$). 

\smallskip

\noindent{\bf{Step 1.b.--}} 
Then we treat the case $\epsilon=-1$, i.e.\ $h(x,y)=(Ax^{-1}, Bx^m y)$. We obtain the following possibilities.
\begin{enumerate}
\item[(a')] All $a_j$ and $b_j$ are equal to $0$; then  $f_1(x,y)=f_2(x,y)=(x^d, y^d)$. 
\item[(b')] $m\geq 0$ and some $a_j$ is not $0$. Then each $b_j$ is a constant, $m=0$, and $B^j a_j=b_j$ for every $j\geq 2$.  In this case, the linear map $(x,y)\mapsto (x, y/B)$ conjugates $f_1$ to $f_2$: $\ell \circ f_1=f_2\circ \ell$.
\item[(c')] $m\leq -1$ and some $a_j$ is not $0$. Comparing the degrees in the two sides of Equation~\eqref{eq:bjaj}, we get $m=-1,-2.$ Thus, from Equation~\eqref{eq:mneq0}, we get $m=-2.$ 
\end{enumerate}
It remains to find a linear projective conjugacy in case (c'). By~\eqref{eq:bjaj} again, there are $c_j\in \bfk, j=2,\dots, d$ such that $a_j(x)=c_jx^j$ and $b_j(x)=c_j(A/B)^jx^j.$
We set $\beta=A/B$ (a root of unity of order dividing $d-1$),  and  write
\begin{align}
	f_1(x,y) &=(x^d : y^d +\sum_{j=2}^d c_j x^{j} y^{d-j})\\ 
	f_2(x,y) &=(x^d : y^d +\sum_{j=2}^d c_j \beta^j x^{j} y^{d-j}).
\end{align}
The initial conjugacy $h[x:y:z]=[Axz : Byz : x^2]$ is not a linear projective automorphism of $\P^2$, but
the automorphism defined by 
$[x:y:z]\mapsto [\beta x: y:z]$ conjugates $f_1$ to $f_2$. As $f_2$ has the same form as in Equation~\eqref{eqpp}, Step (1.a) shows that $f_1$ and $f_2$ are conjugate, some linear automorphism of $\P^2$, to an endomorphism of type~\eqref{eqpps}. Doing so, the conjugacy $h(x,y)=(A/x, By/x^2)$ becomes of type $(A/x, By)$ as in case (b'). 

This concludes the proof in the setting of Step 1.

\smallskip

{\bf{Step 2.--}} The remaining case is when $f_i=(x^{-d}, F_i(x,y))$, for $i=1,2$, with 
\begin{equation}
F_1(x,y)=\sum_{j=0}^d a_j(x)x^{-d} y^{d-j} \; \text{ and } \;  F_2(x,y)=\sum_{j=0}^d b_j(x)x^{-d} y^{d-j}
\end{equation}
for some polynomial functions $a_j, b_j\in \bfk[x]$ satisfying $\deg(a_j)\leq j$, $\deg(b_j)\leq j$, and $a_0b_0\neq 0$.

We first assume that after a conjugacy by an affine transformation of the 
plane, $f_1^2(x,y)=(x^{d^2}, y^{d^2}).$ By the first step, we may assume that $f_2^2(x,y)=(x^{d^2}, y^{d^2})$ after a similar conjugacy. Then, 
$h(x,y)=(Ax, Bx^my)$ with $A$ and $B$ two 
roots of unity of order dividing $d^2-1$ and $m\in \Z\setminus\{0\}$.
The three lines $L_1=\{x=0\}$, $L_2=\{y=0\}$, and $L_3=\{z=0\}$ are totally invariant by both $f_1^2$ and $f_2^2$, and are the only totally invariant irreducible curves. Then for $i=1,2$, $f_i^{-1}$ permutes these three lines and 
$L_1\cup L_3 \cup L_2$ is totally invariant by $f_i$. Hence, after conjugacy by  affine transformations of the 
plane, we may assume that $f_1=f_2=(x^{-d}, x^{-d}y^d)$ (or $(x^d, y^d)$, as in Assertion~(1)).  Writing the conjugacy equation $h\circ f_1=f_2\circ h$, we see that $h$ is in fact linear, $A$ is a root of unity of order dividing $d+1$  and $B^{d-1}=A^{-1}$.

Now we can assume that $f_1^2$ is not conjugate to $(x^{d^2}, y^{d^2})$ by any affine transformation of the 
plane. We apply Step~1 to $f_1^2, f_2^2$ and $h$. Thus, after a conjugacy of each $f_i$ by an affine transformation of the 
plane of the form $(x,y)\mapsto \left( x,y+c \right)$,
we may assume that, up to a permutation of $f_1$ and $f_2$, $h$ takes form $(Ax, Bxy)$, or $(A/x, By)$, or $(Ax^{-1}, Bx^{-2}y)$.
Writing the conjugacy equation $h\circ f_1=f_2\circ h$ and looking at the term of degree $d$ in $y$, we get a contradiction. This concludes the proof.
%
%
%
\end{proof}

\medskip

{\noindent}{\bf{6. Endomorphisms of $(\bbA^1_\bfk\setminus \{0 \})^2$. --}} Denote by $[x:y:z]$ the homogeneous coordinates of $\P^2_\bfk$
and by $(x,y)$ the coordinates of the open subset $V:=(\bbA^1_\bfk\setminus \{0 \})^2$ defined by $xy\neq 0$, $z=1$. 
We write $f_i=a_i\circ g_d$ 
as in case (P3) of Section~{\bf{3}}. 
Since $h$ is an automorphism of $(\bbA^1_\bfk\setminus \{0 \})^2$, it is the composition $t_h\circ m_h$ of a 
diagonal map $t_h(x,y)=(ux,vy)$, for some pair  $(u,v)\in (\bfk^*)^2$, and  a monomial map 
$m_h(x,y)=(x^ay^b, x^cy^d)$, for some matrix
\begin{equation}
M_h:=\left(\begin{array}{lr}a & b \\Â c & d \end{array}\right) \in \GL_2(\Z).
\end{equation}
Also, note that the group $\mathfrak{S}_3\subset \Bir(\P^2_\bfk)$ of permutations of the coordinates $[x:y:z]$ corresponds to a finite 
subgroup $S_3$ of $\GL_2(\Z)$. 

Since $m_h$ commutes with $g_d$ and $g_d \circ t_h=t_h^d\circ g_d$,  the conjugacy equation is equivalent to 
\begin{equation}\label{eq:conj-eq-P3}
t_h\circ (m_h\circ a_1\circ m_h^{-1})\circ (g_d\circ m_h)= a_2\circ t_h^d\circ (g_d\circ m_h).
\end{equation}
The automorphisms $a_1$ and $a_2$ are monomial maps, induced by elements $A_1$ and $A_2$ of $S_3$, 
and Equation~\eqref{eq:conj-eq-P3} implies that {\sl{$M_h$ conjugates $A_1$ to $A_2$ in $\GL_2(\Z)$}}; indeed, 
the matrices can be recovered by looking at the action on the set of units $wx^m y^n$ in $\bfk(V)$ (or on the fundamental group 
$\pi_1(V(\C))$ if $\bfk=\C$). There are two possibilities : 
\begin{itemize}
\item[(a)] either $A_1=A_2=\mathrm{Id}$, there is no constraint on $m_h$;
\item[(b)] or $A_1$ and $A_2$ are non-trivial permutations, they are conjugate by an element $P\in S_3$,  and 
$M_h=\pm A_2^j \circ P$, for some $j\in \Z$. 
\end{itemize}
In both cases, $u$ and $v$ are roots of unity (their orders are determined by $d$ and the $A_i$).
Let $p$ be the monomial transformation associated to $P$; it is a permutation of the coordinates, 
hence an element of $\Aut(\P^2_\bfk)$. Then, $h'(x,y)=t_h\circ p$ is an element of $\Aut(\P^2_\bfk)$ 
that conjugates $f_1$ to $f_2$. 
%
%
%

\medskip

{\noindent}{\bf{7. Centralizers. --}}  Let $f$ be an endomorphism of a complex projective variety $X$. 
Suppose there exists an ample line bundle $L$ on $X$ such that $f^*L$ is isomorphic to $L^{\otimes d}$, for some degree $d\geq 2$ (one says $f$ is polarized). Then, the set of repelling periodic points  of $f$ is Zariski dense in $X$
(see~\cite{Briend-Duval:2001}).

If $H$ is any subgroup of $\Bir(X)$, we set ${\sf{Cent}}_H(f)=\{h\in H\; ; \; hf=fh\}$.

Now, suppose $G$ is an algebraic group acting birationally and faithfully on~$X$. That is, there is a rational map 
$\alpha\colon G\times X\dasharrow X$, $(x,x)\mapsto \alpha(g,x)$ such that (i) its domain of definition projects onto $X$ via the second projection $G\times X\to X$,  (ii) $\alpha(g,\alpha(h,x))=\alpha(gh,x)$, (iii) $\alpha(e_G,x)=x$, and (iv) the induced map $g\in G\mapsto  \alpha(g,\cdot)\in \Bir(X)$ is injective. 
Then, we can identify $G$ to a subgroup of $\Bir(X)$. 
\begin{lem}
With the above assumptions,  {\emph{${\sf{Cent}}_G(f)$ is finite}}.
\end{lem}
\begin{proof} ${\sf{Cent}}_G(f)$ is an algebraic subgroup of $G$, so if it were infinite, it would contain a $1$-parameter subgroup $(g_t)$, with $t$ in the additive or the multiplicative group. Taking derivatives, the equation $f\circ g_t=g_t\circ f$ provides a non-zero rational vector field $v\colon x\mapsto v(x)\in T_xX$ such that $f_*v=v$ (i.e. $v(f(x))=Df_xv(x)$). 
If $z$ is a periodic point of $f$ of period $q$ in the open set where $v$ is well defined and not zero, $v(z)$ would be an eigenvector of $Df_z^q$ with eigenvalue $1$; so, $z$ could not be repelling, a contradiction with the density of repelling periodic points. \end{proof}

Now, consider an endomorphism $f$ of $\P^2$ of degree $d\geq 2$, over a field $\bfk$ of characteristic $0$. The previous lemma implies that its centralizer in $\PGL_3(\bfk)$ is finite. To describe its centralizer in $\Bir(\P^2_\bfk)$, it remains to study cases (1) and (2) in Theorem~A. 

From Section~6, if $f[x:y:z]=[x^d:y^d:z^d]$, its centralizer is the composition of any birational monomial map with a diagonal map $[x:y:z]\mapsto [Ax:By:z]$ whose coefficients $A$ and $B$ are roots of unity of order dividing $d-1$. If $f$ is a composition of $[x:y:z]\mapsto [x^d:y^d:z^d]$ by a non-trivial permutaion of the coordinates, and $h$ commutes to $f$, then $h$ commutes to $f^2$ and $f^3$, hence it must be such a composition of a monomial and a diagonal map; but then an easy computation shows that $h$ must in fact be diagonal. 

From Section~5, if $f$ is an element from case (2) of Theorem~A, we can conjugate $f$ to 
\begin{equation}\label{eq:strange_case}
f(x,y)= (x^d, y^d +\sum_{j=2}^d a_j y^{d-j})
\end{equation}
for some constants $a_j$, not all $0$, and  then its centralizer in $\Bir(\P^2_\bfk)$  is the semi-direct product 
\begin{equation}\label{eq:strange_case_centralizer}
\Z/2\Z\ltimes E(f)
\end{equation}
where $\Z/2\Z$ is generated by the involution $h(x,y)=(1/x, y)$ and $E(f)$ is the group of diagonal transformations 
$(x,y)\mapsto (Ax,By)$ with $A^{d-1}=1$, $B^{d-1}=1$, and $B^{j-1}=1$ for all indices $j$ such that $a_j\neq 0$. 
Thus, we get 
\begin{cor}
Let $\bfk$ be an algebraically closed field of characteristic $0$. Let $f$ be an endomorphism of $\P^2_\bfk$ of degree $d\geq 2$. 
\begin{itemize}
\item If $f[x:y:z]=[x^d:y^d:z^d]$, its centralizer is the semi-direct product $
\GL_2(\Z)\ltimes D(d-1)
$
where $D(d-1)\subset\PGL_3(\bfk)$ is the finite group of diagonal transformations $[x:y:z]\mapsto [Ax:By:z]$ such that $A^{d-1}=1=B^{d-1}$.

\item If $f$ is the endomorphism from Equation~\eqref{eq:strange_case}, then its centralizer in $\Bir(\P^2)$ is the finite group described in Equation~\eqref{eq:strange_case_centralizer}.

\item If $f$ is not conjugate to an endomorphism from the two previous items by an element of $\PGL_3(\bfk)$, then its centralizer in $\Bir(\P^2_\bfk)$ is a finite subgroup of $\PGL_3(\bfk)$.
\end{itemize}
\end{cor}

{\noindent}{\bf{8. An example in positive characteristic. --}}  Assume that $q=p^s$ with $s\geq 2$ (resp.\ $s\geq 3$ if $p=2$).
Set $G:=xy^p+(x-1)y$. Then, 
$$f_1(x,y)=(x^q, y^q+G(x,y))$$
defines an endomorphism of $\bbA^2$ that extends to an endomorphism of $\P^2$.

Consider a polynomial $P(x)\in\bfF_q[x]$ such that  $2 \leq \deg(P) \leq \frac{q}{p}-1$.
Observe that $\deg (G) < \deg (G(x,y+P(x))) < q$.
Then $g(x,y)= (x,y-P(x))$
is an automorphism of $\bbA^2_\bfk$ that conjugates $f_1$ to 
\begin{align}f_2(x,y) &:=g\circ f_1\circ g^{-1}(x,y) \notag \\
&=(x^q,y^q+P(x)^q+G(x,y+P(x))-P(x^q))\\
&=(x^q,y^q+G(x,y+P(x))).\notag
\end{align}
Just like $f_1$, $f_2$ is an endomorphism of $\bbA^2$ that extends to a regular endomorphism of~$\P^2$
(here we use the inequality $\deg (G(x,y+P(x))) < q$).

Let us prove that $f_1$ and $f_2$ are not conjugate by any automorphism of $\P^2$.
We assume that there exists $h\in \PGL_3(\overline{\bfF_q})$ such that $h\circ f_1=f_2\circ h$ and seek a contradiction.
Consider the pencils of lines through the point $[0:1:0]$ in $\P^2$; for $a\in \bfF_q$ we denote by 
$L_a$ the line $\{x=az\}$, and by $L_\infty$ the line $\{ z=0\}$. 
Then 
\begin{align}\{L_a \; ; \; a\in \bfF_q\cup \{\infty\}\} & = \{{\text{lines}} \ L \ {\text{such that}}\  f_1^{-1}L=L\} \\
&  =\{{\text{lines}} \ L \ {\text{such that}}\ f_2^{-1}L=L\};
\end{align}                  
in other words, the lines $L_a$ for $a\in \bfF_q\cup\{\infty\}$ are exactly the lines which are totally invariant under
the action of $f_1$ (resp. of $f_2$).
Since $h$ conjugates $f_1$ to $f_2$, it permutes these lines. In particular, $h$ fixes the point $[0:1:0]$, and if 
we identify $L_a\cap \bbA^2$ with $\bbA^1$  using the parametrization $y\mapsto (a,y) $ then 
$h$ maps $L_a$ to another line $L_{a'}$ in an affine way: $h(a,y)=(a',\alpha y +\beta)$.

Since $g$ conjugates $f_1$ to $f_2$ and $g$ fixes each of the lines $L_a$, we know that $f_1|_{L_a}$ is conjugate to $f_2|_{L_a}$
for every  $a\in \bfF_q$; for $a=\infty$, both $f_1|_{L_\infty}$ and $f_2|_{L_\infty}$ are conjugate to $y\mapsto y^q$. Moreover
\begin{itemize}
\item $a=\infty$ is the unique parameter  such that $f_1|_{L_a}$  is conjugate to $y\mapsto y^q$ by an affine map $y \mapsto \alpha y + \beta$;
\item $a=0$ is the unique parameter such that $f_1|_{L_a}$ is conjugate to $y\mapsto y^q-y$ by an affine map;
\item $a=1$  is the unique parameter such that $f_1|_{L_a}$ is conjugate to $y\mapsto y^q+y^p$ by an affine map.
\end{itemize}
And the same properties hold for $f_2$.
As a consequence, we obtain $h(L_{\infty})=L_{\infty}$, $h(L_0)=L_0$ and $h(L_1)=L_1$; this means that there are coefficients  $\alpha \in {\overline{\bfF_q}}^*$ and $\beta, \gamma\in {\overline{\bfF_q}}$ such that 
$h(x,y)= (x, \alpha y+\beta x+\gamma)$.
Writing down the relation $h\circ f_1=f_2\circ h$ we obtain the relation
\begin{align}
\alpha y^q+\alpha G(x,y)+\beta x^q+\gamma = &\;  \alpha^q y^q+\beta^qx^q+\gamma^q\\
&+ G(x, \alpha y+\beta x+\gamma +P(x)).
\end{align}
We note that $1<\deg G(x,y)<\deg G(x, \alpha y+\beta x+\gamma +P(x))<q$.
Comparing the terms of degree $q$, we get $\alpha y^q+\beta x^q=\alpha^q y^q+\beta^qx^q.$
It follows that 
\begin{align}
\alpha G(x,y)+\gamma =  \gamma^q+ G(x, \alpha y+\beta x+\gamma +P(x)).
\end{align}
Then $\deg G(x,y)=\deg G(x, \alpha y+\beta x+\gamma +P(x))$, which is a contradiction.

\bibliographystyle{plain}
\bibliography{dd}

\end{document}